\documentclass[12pt]{amsart}
\usepackage{amsmath}
\usepackage{amssymb}
\usepackage[all]{xy}
\usepackage{longtable}
\allowdisplaybreaks[1]
\newtheorem{theorem}[equation]{Theorem}
\newtheorem{lemma}[equation]{Lemma}
\newtheorem{proposition}[equation]{Proposition}
\newtheorem{corollary}[equation]{Corollary}

\theoremstyle{definition}

\theoremstyle{remark}

\numberwithin{equation}{section}

\DeclareMathAlphabet{\matheur}{U}{eur}{m}{n}
\newcommand{\dotcup}{\ensuremath{\mathaccent\cdot\cup}}

\DeclareMathOperator{\Gal}{Gal}

\DeclareMathOperator{\Li}{Li}

\newcommand*\pFqskip{8mu}
\catcode`,\active
\newcommand*\pFq{\begingroup
        \catcode`\,\active
        \def ,{\mskip\pFqskip\relax}%
        \dopFq
}
\catcode`\,12
\def\dopFq#1#2#3#4#5{%
        {}_{#1}F_{#2}\biggl(\genfrac..{0pt}{}{#3}{#4};#5\biggr)%
        \endgroup
}
\renewcommand{\Im}{\operatorname{Im}}
\renewcommand{\Re}{\operatorname{Re}}
\begin{document}

\title[Three-variable Mahler measures]{Three-variable Mahler measures and special values of modular and Dirichlet $L$-series}

\author{Detchat Samart}
\address{Department of Mathematics, Texas A{\&}M University, College
Station,
TX 77843, USA} \email{detchats@math.tamu.edu}

\thanks{The auther's research was partially supported by NSF Grant DMS-0903838}


\date{August 19, 2012}

\begin{abstract}
In this paper we prove that the Mahler measures of the Laurent polynomials $(x+x^{-1})(y+y^{-1})(z+z^{-1})+k^{1/2}$, $(x+x^{-1})^2(y+y^{-1})^2(1+z)^3z^{-2}-k$, and $x^4+y^4+z^4+1+k^{1/4}xyz$, for various values of $k$, are of the form $r_1 L'(f,0)+r_2 L'(\chi,-1)$, where $r_1,r_2\in \mathbb{Q}$, $f$ is a CM newform of weight $3$, and $\chi$ is a quadratic character. Since it has been proved that these Mahler measures can also be expressed in terms of logarithms and $_5F_4$-hypergeometric series, we obtain several new hypergeometric evaluations and transformations from these results.
\end{abstract}


\maketitle
\section{Introduction}\label{Sec:Introduction}
The logarithmic Mahler measure of a Laurent 
polynomial $P\in \mathbb{C}[X_1^{\pm 1},\ldots,X_n^{\pm 1}]$ is defined by 
$$m(P)=\int_0^1\cdots \int_0^1 \log |P(e^{2\pi i \theta_1},\ldots,e^{2\pi i \theta_n})|\,d\theta_1\cdots d\theta_n.$$
It has been shown that Mahler measures of certain types of Laurent polynomials are related to special values of  $L$-functions. One of the first examples, proved by Smyth \cite{Smyth}, is 
\begin{equation*}
m(x+y+1)=\frac{3\sqrt{3}}{4\pi}L(\chi_{-3},2)=L'(\chi_{-3},-1),
\end{equation*}
where here and throughout $\chi_{D}(n)=\left(\frac{D}{n}\right).$ 

Boyd \cite{Boyd} has made substantial progress in this research area by showing that Mahler measures of a number of families of two-variable polynomials are numerically equal to rational multiples of $L'(E,0)$, where $E$ is the elliptic curve over $\mathbb{Q}$ given by the corresponding polynomial. These investigations resulted in a large number of potential conjectured formulas, and this significant discovery has motivated many others to find proofs of these formulas and explanations for this phenomenon. For instance, Rodriguez Villegas rigorously verified that such formulas hold for many tempered polynomials whose corresponding elliptic curves have complex multiplication by some number fields (for a definition of tempered polynomials see \cite[\S III]{RV}). More recently, Rogers and Zudilin \cite{RZ} have proved an early conjecture of Deninger \cite{Deninger} that 
\begin{equation*}
m(1+X+X^{-1}+Y+Y^{-1})=\frac{15}{4\pi^2}L(E_{15},2),
\end{equation*}
where $E_{15}$ is an elliptic curve of conductor $15$.\\
\indent Note that if $E$ is an elliptic curve over $\mathbb{Q}$, then by the celebrated modularity theorem $$L(E,s)=L(f,s)$$ for some newform $f$ of weight $2$. Therefore, it is interesting to look for examples of polynomials in more variables whose Mahler measures are related to $L$-values of higher dimensional varieties or those of modular forms of higher weights corresponding to those polynomials.

In the three-variable case, Bertin \cite{Bertin} proved that certain $P\in\mathbb{C}[X^{\pm 1},Y^{\pm 1},Z^{\pm 1}]$ have Mahler measures of the form
\begin{equation*}
m(P)=r\left(\frac{\sqrt{N}}{2\pi}\right)^3L(g,3),
\end{equation*}
where $r\in\mathbb{Q}$ and where $g$ is a Hecke newform of weight $3$ for $\Gamma_0(N)$. The zero locus of $P$ defines a singular $K3$ surface (having Picard number 20), and $L(g,s)$ appears as a factor in its $L$-series. This can be considered as an analogue of the two-variable case, where the modularity theorem for elliptic curves over $\mathbb{Q}$ is replaced by the modularity theorem for singular $K3$ surfaces, originally proved by Livn\'{e} \cite{Liv}. Rogers \cite{RogersMain} then extended these results by showing that Mahler measures of the polynomials given in \cite{Bertin} can be written as linear combinations of Mahler measures of some other polynomials which are of hypergeometric type. In other words, the latter Mahler measures, of a family of polynomials parametrized by $k$, are of the form 
\begin{equation*}
m(P_k)=\Re\left(\log(k)+\frac{r_1}{k}\pFq{p}{q}{a_1,a_2,\ldots,a_p}{b_1,b_2,\ldots,b_q}{\frac{r_2}{k}}\right),
\end{equation*}
where $r_1,r_2\in \mathbb{Q}$ and 
\begin{equation*}
\pFq{p}{q}{a_1,a_2,\ldots,a_p}{b_1,b_2,\ldots,b_q}{x}=\sum_{n=0}^{\infty}\frac{(a_1)_n\cdots(a_p)_n}{(b_1)_n\cdots(b_q)_n}\frac{x^n}{n!}
\end{equation*}
with $(c)_n=\Gamma(c+n)/\Gamma(c)$. Combining these results together he deduced interesting formulas relating hypergeometric values to special values of modular $L$-series, including
\begin{equation}\label{hyperI}
\pFq{5}{4}{\frac{5}{4},\frac{3}{2},\frac{7}{4},1,1}{2,2,2,2}{1}=\frac{256}{3}\log(2)-\frac{5120\sqrt{2}}{3\pi^3}L(f,3)
\end{equation}
and
\begin{equation}\label{hyperII}
\pFq{5}{4}{\frac{4}{3},\frac{3}{2},\frac{5}{3},1,1}{2,2,2,2}{1}=18\log(2)+27\log(3)-\frac{810\sqrt{3}}{\pi^3}L(g,3),
\end{equation}
where $f(\tau)=\eta(\tau)^2\eta(2\tau)\eta(4\tau)\eta(8\tau)^2$ and $g(\tau)=\eta(2\tau)^3\eta(6\tau)^3$ (see \cite[Cor.~2.6]{RogersMain}). (Here $\eta(\tau)$ is the Dedekind eta function.)

Following Rogers' notations, we denote 
\begin{align*}
f_2(k)&:=2m\left((x+x^{-1})(y+y^{-1})(z+z^{-1})+k^{1/2}\right),\\
f_3(k)&:=m\left((x+x^{-1})^2(y+y^{-1})^2(1+z)^3z^{-2}-k\right),\\
f_4(k)&:=4m\left(x^4+y^4+z^4+1+k^{1/4}xyz\right),
\end{align*}
with the parameter $k\in\mathbb{C}$. (Note that the Mahler measures $f_2(k)$ and $f_4(k)$ do not depend on the choices of the square root and the fourth root of $k$, respectively.) These Mahler measures are also known to be of hypergeometric type by the following result:  
\begin{proposition}\label{P:Rogers}[{Rogers \cite[Prop.~2.2]{RogersMain}}]
\begin{enumerate}
\item[(i)]
If $|k|\geq 64$, then 
$\displaystyle f_2(k)=\Re\left(\log(k)-\frac{8}{k}\pFq{5}{4}{\frac{3}{2},\frac{3}{2},\frac{3}{2},1,1}{2,2,2,2}{\frac{64}{k}}\right).
$
\item[(ii)]
If $|k|\geq 128$, then 
$\displaystyle
f_3(k)=\Re\left(\log(k)-\frac{12}{k}\pFq{5}{4}{\frac{4}{3},\frac{3}{2},\frac{5}{3},1,1}{2,2,2,2}{\frac{108}{k}}\right).
$
\item[(iii)]
If $|k|\geq 256$, then 
$\displaystyle
f_4(k)=\Re\left(\log(k)-\frac{24}{k}\pFq{5}{4}{\frac{5}{4},\frac{3}{2},\frac{7}{4},1,1}{2,2,2,2}{\frac{256}{k}}\right).
$
\end{enumerate}
\end{proposition}

The ultimate goal of this paper is to prove:
\begin{theorem}\label{T:Main} 
The following equalities hold:
\begin{align} 
f_2(64)&=\frac{128}{\pi^3}L(h,3)=8L'(h,0), \label{A:64}\\
f_2(256)&=\frac{64\sqrt{3}}{\pi^3}L(g_{48},3)+\frac{16}{3\pi}L(\chi_{-4},2)=\frac{4}{3}(L'(g_{48},0)+2L'(\chi_{-4},-1)), \label{A:256}\\
f_3(216)&=\frac{45\sqrt{6}}{\pi^3}L(g_{24}^{(1)},3)+\frac{45\sqrt{3}}{16\pi}L(\chi_{-3},2)=\frac{15}{4}(L'(g_{24}^{(1)},0)+L'(\chi_{-3},-1)), \label{A:216}\\
f_3(1458)&=\frac{405\sqrt{3}}{4\pi^3}L(g,3)+\frac{15}{2\pi}L(\chi_{-4},2)=\frac{15}{8}(9L'(g,0)+2L'(\chi_{-4},-1)), \label{A:1458}\\
f_4(648)&=\frac{160}{\pi^3}L(h,3)+\frac{5}{\pi}L(\chi_{-4},2)=\frac{5}{2}(4L'(h,0)+L'(\chi_{-4},-1)),\label{A:648}\\
f_4(2304)&=\frac{80\sqrt{6}}{\pi^3}L(g_{24}^{(2)},3)+\frac{5\sqrt{3}}{\pi}L(\chi_{-3},2)=\frac{20}{3}(L'(g_{24}^{(2)},0)+L'(\chi_{-3},-1)),\label{A:2304}\\
f_4(20736)&=\frac{80\sqrt{10}}{\pi^3}L(g_{40},3)+\frac{32\sqrt{2}}{5\pi}L(\chi_{-8},2)=\frac{4}{5}(5L'(g_{40},0)+2L'(\chi_{-8},-1)),\label{A:20736}\\
f_4(614656)&=\frac{800\sqrt{2}}{3\pi^3}L(f,3)+\frac{10\sqrt{3}}{\pi}L(\chi_{-3},2)=\frac{40}{3}(5L'(f,0)+L'(\chi_{-3},-1)),\label{A:614656}\end{align}
where 
\begin{align*}
f(\tau)&=\eta(\tau)^2\eta(2\tau)\eta(4\tau)\eta(8\tau)^2,\\
g(\tau)&=\eta(2\tau)^3\eta(6\tau)^3,\\
h(\tau)&=\eta(4\tau)^6,\\ g_{48}(\tau)&=\frac{\eta(4\tau)^9\eta(12\tau)^9}{\eta(2\tau)^3\eta(6\tau)^3\eta(8\tau)^3\eta(24\tau)^3},\\ g_{24}^{(1)}(\tau)&=q+2q^2-3q^3+4q^4-2q^5-6q^6-10q^7+8q^8+9q^9-4q^{10}+\cdots,\\ g_{24}^{(2)}(\tau)&=q-2q^2+3q^3+4q^4+2q^5-6q^6-10q^7-8q^8+9q^9-4q^{10}-\cdots,\\
g_{40}(\tau)&=q-2q^2+4q^4+5q^5+6q^7-8q^8+9q^9-10q^{10}-18q^{11}-6q^{13}-\cdots.
\end{align*}
\end{theorem}

We see from \cite{DKM} that $f, g,$ and $h$ defined above are newforms with complex multiplication (CM); i.e., the newforms are the inverse Mellin transforms of Hecke $L$-series, in $S_3(\Gamma_0(8),\chi_{-8}), S_3(\Gamma_0(12),\chi_{-3})$, and $S_3(\Gamma_0(16),\chi_{-4})$, respectively. Also, we will see in the next section that $g_{24}^{(1)},g_{24}^{(2)}\in S_3(\Gamma_0(24),\chi_{-24})$ and  $g_{40}\in S_3(\Gamma_0(40),\chi_{-40})$. Moreover, they all are newforms of CM type. On the other hand, it follows immediately by \cite[Thm.~1.64]{Ono} that $g_{48} \in S_3(\Gamma_0(48),\chi_{-3})$. Computing some first Fourier coefficients yields
\begin{align*}
g_{48}(\tau)&=q+3q^3-2q^7+9q^9-22q^{13}-26q^{19}-6q^{21}+\cdots, \\
g(\tau)&=q-3q^3+2q^7+9q^9-22q^{13}+26q^{19}-6q^{21}+\cdots;
\end{align*} that is, $g_{48}$ is a twist of $g$ by $\chi_{-4},$ so $g_{48}$ is also a CM newform. It might be worth pointing out that although $g_{24}^{(1)}$ and $g_{24}^{(2)}$ cannot be represented by an eta quotient, we can write them as linear combinations of eta quotients which form a basis for $S_3(\Gamma_0(24),\chi_{-24})$. However, this fact will not be used to prove \eqref{A:216} and \eqref{A:2304}.
Applying Proposition~\ref{P:Rogers} together with Theorem~\ref{T:Main} one can easily deduce many formulas similar to \eqref{hyperI} and \eqref{hyperII}.
\begin{corollary}\label{C:Hyper}
Let $f,g,h,g_{48},g_{24}^{(1)},g_{24}^{(2)},$ and $g_{40}$ be as defined in Theorem~\ref{T:Main}. Then the following formulas hold:
\begin{align*}
\pFq{5}{4}{\frac{3}{2},\frac{3}{2},\frac{3}{2},1,1}{2,2,2,2}{1}&=48\log(2)-64L'(h,0),\\
\pFq{5}{4}{\frac{3}{2},\frac{3}{2},\frac{3}{2},1,1}{2,2,2,2}{\frac{1}{4}}&=256\log(2)-\frac{128}{3}\left(L'(g_{48},0)+2L'(\chi_{-4},-1)\right),\\
\pFq{5}{4}{\frac{4}{3},\frac{3}{2},\frac{5}{3},1,1}{2,2,2,2}{\frac{1}{2}}&=54\log(6)-\frac{135}{2}\left(L'(g_{24}^{(1)},0)+L'(\chi_{-3},-1)\right),\\
\pFq{5}{4}{\frac{4}{3},\frac{3}{2},\frac{5}{3},1,1}{2,2,2,2}{\frac{2}{27}}&=\frac{243}{2}\log(2)+729\log(3)-\frac{3645}{16}\left(9L'(g,0)+2L'(\chi_{-4},-1)\right),\\
\pFq{5}{4}{\frac{5}{4},\frac{3}{2},\frac{7}{4},1,1}{2,2,2,2}{\frac{32}{81}}&=81\log(2)+108\log(3)-\frac{135}{2}\left(4L'(h,0)+L'(\chi_{-4},-1)\right),\\
\pFq{5}{4}{\frac{5}{4},\frac{3}{2},\frac{7}{4},1,1}{2,2,2,2}{\frac{1}{9}}&=768\log(2)+192\log(3)-640\left(L'(g_{24}^{(2)},0)+L'(\chi_{-3},-1)\right),\\
\pFq{5}{4}{\frac{5}{4},\frac{3}{2},\frac{7}{4},1,1}{2,2,2,2}{\frac{1}{81}}&=6912\log(2)+3456\log(3)-\frac{3456}{5}\left(5L'(g_{40},0)+2L'(\chi_{-8},-1)\right),\\
\pFq{5}{4}{\frac{5}{4},\frac{3}{2},\frac{7}{4},1,1}{2,2,2,2}{\frac{1}{2401}}&=\frac{614656}{3}\log(2)+\frac{307328}{3}\log(7)\\&\qquad-\frac{3073280}{9}\left(5L'(f,0)+L'(\chi_{-3},-1)\right).
\end{align*}
\end{corollary} 
Furthermore, the following hypergeometric transformations are immediate consequences of \eqref{hyperI}, \eqref{hyperII}, and Corollary~\ref{C:Hyper}:
\begin{align*}
\pFq{5}{4}{\frac{5}{4},\frac{3}{2},\frac{7}{4},1,1}{2,2,2,2}{1}&=\frac{3}{12005}\pFq{5}{4}{\frac{5}{4},\frac{3}{2},\frac{7}{4},1,1}{2,2,2,2}{\frac{1}{2401}}+\frac{512}{15}\log(2)-\frac{128}{5}\log(7)\\&\qquad+\frac{256}{3}L'(\chi_{-3},-1),\\
\pFq{5}{4}{\frac{4}{3},\frac{3}{2},\frac{5}{3},1,1}{2,2,2,2}{1}&=\frac{16}{243}\pFq{5}{4}{\frac{4}{3},\frac{3}{2},\frac{5}{3},1,1}{2,2,2,2}{\frac{2}{27}}+10\log(2)-21\log(3)+30L'(\chi_{-4},-1),\\
\pFq{5}{4}{\frac{3}{2},\frac{3}{2},\frac{3}{2},1,1}{2,2,2,2}{1}&=\frac{32}{135}\pFq{5}{4}{\frac{5}{4},\frac{3}{2},\frac{7}{4},1,1}{2,2,2,2}{\frac{32}{81}}+\frac{144}{5}\log(2)-\frac{128}{5}\log(3)+16L'(\chi_{-4},-1).
\end{align*}
Note that $L'(\chi_{-4},-1)=\frac{2}{\pi}G$, where $G$ is the \textit{Catalan's constant}. Therefore, we also obtain new representations of $G$ in terms of $_5F_4$-hypergeometric series.

\section{Proof of The Main Theorem} \label{sec:main}
Throughout this paper, $q$ will be a function of $\tau\in \mathbb{C}$ with $\Im(\tau)>0$ given by $q:=q(\tau)=e^{2\pi i \tau}$, and we let $\displaystyle\sideset{}{'}\sum_{m,n\in \mathbb{Z}}$ denote the summation over $m,n\in\mathbb{Z}$ with $(m,n)\neq (0,0).$ As usual, we denote $$\eta(\tau):=q^{\frac{1}{24}}\prod_{n=1}^{\infty}(1-q^n) \quad\text{ and }\quad \Delta(\tau):=\eta(\tau)^{24}.$$ To prove Theorem~\ref{T:Main} we first prove a more general result stating that $f_2(k),f_3(k),$ and $f_4(k)$, for some values of $k$, can be expressed as Eisenstein-Kronecker series. 
\begin{proposition} \label{P:general}
Assume that $q\in (0,1)$, and let  
\begin{align*} s_2(q)&=-\frac{\Delta\left(\tau+\frac{1}{2}\right)}{\Delta(2\tau+1)},\\
s_3(q)&=\left(27\left(\frac{\eta(3\tau)}{\eta(\tau)}\right)^6+\left(\frac{\eta(\tau)}{\eta(3\tau)}\right)^6\right)^2,\\
s_4(q)&=\frac{\Delta(2\tau)}{\Delta(\tau)}\left(16\left(\frac{\eta(\tau)\eta(4\tau)^2}{\eta(2\tau)^3}\right)^4+\left(\frac{\eta(2\tau)^3}{\eta(\tau)\eta(4\tau)^2}\right)^4\right)^4.
\end{align*}
\begin{enumerate}
\item[(i)]
If $\displaystyle\Im(\tau)\geq \frac{1}{2}$, then
\begin{multline*}
f_2(s_2(q))=\frac{2\Im(\tau)}{\pi^3}\sideset{}{'}\sum_{m,n\in \mathbb{Z}}\biggl(-\left(\frac{4(m\Re(\tau)+n)^2}{[(m\tau+n)(m\bar{\tau}+n)]^3}-\frac{1}{[(m\tau+n)(m\bar{\tau}+n)]^2}\right)\\
+16\left(\frac{4(4m\Re(\tau)+n)^2}{[(4m\tau+n)(4m\bar{\tau}+n)]^3}-\frac{1}{[(4m\tau+n)(4m\bar{\tau}+n)]^2}\right)\biggr).
\end{multline*}

\item[(ii)]
If $\displaystyle\Im(\tau)\geq \frac{1}{\sqrt{3}}$, then
\begin{multline*}
f_3(s_3(q))=\frac{15\Im(\tau)}{4\pi^3}\sideset{}{'}\sum_{m,n\in \mathbb{Z}}\biggl(-\left(\frac{4(m\Re(\tau)+n)^2}{[(m\tau+n)(m\bar{\tau}+n)]^3}-\frac{1}{[(m\tau+n)(m\bar{\tau}+n)]^2}\right)\\
+9\left(\frac{4(3m\Re(\tau)+n)^2}{[(3m\tau+n)(3m\bar{\tau}+n)]^3}-\frac{1}{[(3m\tau+n)(3m\bar{\tau}+n)]^2}\right)\biggr).
\end{multline*}

\item[(iii)]
If $\displaystyle\Im(\tau)\geq \frac{1}{\sqrt{2}}$, then
\begin{multline*}
f_4(s_4(q))=\frac{10\Im(\tau)}{\pi^3}\sideset{}{'}\sum_{m,n\in \mathbb{Z}}\biggl(-\left(\frac{4(m\Re(\tau)+n)^2}{[(m\tau+n)(m\bar{\tau}+n)]^3}-\frac{1}{[(m\tau+n)(m\bar{\tau}+n)]^2}\right)\\
+4\left(\frac{4(2m\Re(\tau)+n)^2}{[(2m\tau+n)(2m\bar{\tau}+n)]^3}-\frac{1}{[(2m\tau+n)(2m\bar{\tau}+n)]^2}\right)\biggr).
\end{multline*}
\end{enumerate}
\end{proposition}

The following lemma gives us some evaluations of $s_2(q),s_3(q),$ and $s_4(q)$ which will be used later in this section.
\begin{lemma}\label{L:s2}
Let $s_2(q),s_3(q),$ and $s_4(q)$ be as defined in Proposition~\ref{P:general}. Then 
\begin{align*}
s_2\left(q\left(\frac{\sqrt{-1}}{2}\right)\right)&=64,  &s_2\left(q\left(\frac{\sqrt{-3}}{2}\right)\right)&=256, \\ s_3\left(q\left(\frac{\sqrt{-3}}{3}\right)\right)&=108, &s_3\left(q\left(\frac{\sqrt{-6}}{3}\right)\right)&=216,  &s_3\left(q\left(\frac{\sqrt{-12}}{3}\right)\right)&=1458,\\
s_4\left(q\left(\frac{\sqrt{-2}}{2}\right)\right)&=256, 
&s_4\left(q\left(\frac{\sqrt{-4}}{2}\right)\right)&=648,
&s_4\left(q\left(\frac{\sqrt{-6}}{2}\right)\right)&=2304, \\
s_4\left(q\left(\frac{\sqrt{-10}}{2}\right)\right)&=20736,
&s_4\left(q\left(\frac{\sqrt{-18}}{2}\right)\right)&=614656.
\end{align*}
\end{lemma}
\begin{proof}
Let us consider the following two Weber modular functions:
\begin{align*} 
\mathfrak{f}(\tau)&:=e^{-\frac{\pi i}{24}}\frac{\eta\left(\frac{\tau+1}{2}\right)}{\eta(\tau)},\\
\mathfrak{f}_1(\tau)&:=\frac{\eta\left(\frac{\tau}{2}\right)}{\eta(\tau)}.
\end{align*}
Weber listed a number of special values of these functions in \cite[p.~721]{Weber}, including 
\begin{align*}
\mathfrak{f}\left(\sqrt{-1}\right)&=2^{\frac{1}{4}}, &\mathfrak{f}\left(\sqrt{-3}\right)&=2^{\frac{1}{3}},\\
\mathfrak{f}_1\left(\sqrt{-2}\right)&=2^{\frac{1}{4}},
&\mathfrak{f}_1\left(\sqrt{-8}\right)^8&=8+8\sqrt{2},\\
\mathfrak{f}_1\left(\sqrt{-4}\right)&=8^{\frac{1}{8}},
&\mathfrak{f}_1\left(\sqrt{-16}\right)^4&=2^{\frac{7}{4}}(1+\sqrt{2}),\\
\mathfrak{f}_1\left(\sqrt{-6}\right)^6&=4+2\sqrt{2},
&\mathfrak{f}_1\left(\sqrt{-24}\right)^{24}&=2^9\left(1+\sqrt{2}\right)^2\left(2+\sqrt{3}\right)^3\left(\sqrt{2}+\sqrt{3}\right)^3,\\
\sqrt{2}\mathfrak{f}_1\left(\sqrt{-10}\right)^2&=1+\sqrt{5},
&\mathfrak{f}_1\left(\sqrt{-40}\right)^8&=2\left(1+\sqrt{5}\right)^2\left(1+\sqrt{2}\right)^2\left(3+\sqrt{10}\right),\\
\mathfrak{f}_1\left(\sqrt{-18}\right)^3&=2^{\frac{3}{4}}\left(\sqrt{2}+\sqrt{3}\right),&\mathfrak{f}_1\left(\sqrt{-72}\right)^{24}&=2^7\left(2+\sqrt{6}\right)^4\left(1+\sqrt{2}\right)^9\left(2+\sqrt{3}\right)^6,\\
\mathfrak{f}_1\left(\sqrt{-12}\right)^4&=2^{\frac{7}{6}}\left(1+\sqrt{3}\right).
\end{align*}
(Actually, there are some typographical errors in the original table containing these values, which were corrected later by Brillhart and Morton \cite{BM}.)\\
Since $\Delta(\tau)$ is a modular form for the full modular group $\Gamma(1)$, we have immediately that $$s_2(q)=\mathfrak{f}(2\tau)^{24},$$
so the first two equalities in the lemma follow easily. 
Note also that 
$$\eta\left(-\frac{1}{\tau}\right)=\sqrt{-i\tau}\eta(\tau).$$ 
Hence 
\begin{equation*}
\frac{\eta\left(\sqrt{-3}\right)}{\eta\left(\frac{\sqrt{-3}}{3}\right)}=\frac{1}{3^\frac{1}{4}},\qquad
\frac{\eta\left(\sqrt{-6}\right)}{\eta\left(\frac{\sqrt{-6}}{3}\right)}=\left(\frac{2}{3}\right)^{\frac{1}{4}}\frac{\eta\left(\sqrt{-6}\right)}{\eta\left(\frac{\sqrt{-6}}{2}\right)}=\left(\frac{2}{3}\right)^{\frac{1}{4}}\frac{1}{\mathfrak{f}_1\left(\sqrt{-6}\right)},
\end{equation*}
and
\begin{equation*}
\frac{\eta\left(\sqrt{-12}\right)}{\eta\left(\frac{\sqrt{-12}}{3}\right)}=\left(\frac{2}{\sqrt{3}}\right)^{\frac{1}{2}}\frac{\eta\left(\sqrt{-12}\right)}{\eta\left(\frac{\sqrt{-12}}{4}\right)}=\left(\frac{2}{\sqrt{3}}\right)^{\frac{1}{2}}\frac{1}{\mathfrak{f}_1\left(\sqrt{-12}\right)\mathfrak{f}_1\left(\sqrt{-3}\right)}=\left(\frac{2}{\sqrt{3}}\right)^{\frac{1}{2}}\frac{\mathfrak{f}\left(\sqrt{-3}\right)}{\mathfrak{f}_1\left(\sqrt{-12}\right)^2},
\end{equation*} where the last equality follows from the relation $$\mathfrak{f}_1(2\tau)=\mathfrak{f}(\tau)\mathfrak{f}_1(\tau).$$ These enable us to evaluate $\displaystyle s_3\left(q(\tau)\right)$ for $\tau\in\left\{\frac{\sqrt{-3}}{3},\frac{\sqrt{-6}}{3},\frac{\sqrt{-12}}{3}\right\}.$ \\
Finally, observe that for every $m\in\mathbb{N}$
\begin{equation*}
s_4\left(q\left(\frac{\sqrt{-m}}{2}\right)\right)=\frac{1}{\mathfrak{f}_1(\sqrt{-m})^{24}}\left(16\frac{\mathfrak{f}_1(\sqrt{-m})^4}{\mathfrak{f}_1(\sqrt{-4m})^8}+\frac{\mathfrak{f}_1(\sqrt{-4m})^8}{\mathfrak{f}_1(\sqrt{-m})^4}\right)^4.
\end{equation*}
Using Weber's results above, one can check in a straightforward manner that the evaluations of $s_4(q)$ in the lemma hold.
\end{proof}

\begin{proof}[Proof of Proposition~\ref{P:general}]
We prove this proposition mainly using the method due to Bertin~ \cite{BertinMain}.  Assume that $t:=\Im(\tau)\geq 1/2$. Since $q$ is real, $|s_2(q)|\geq s_2\left(e^{-\pi}\right)=64$ by Lemma~\ref{L:s2}. Analyzing the proof of \cite[Thm. ~2.3]{RogersMain}, one sees that the corresponding $|q|=e^{-2\pi t}$ is small enough to imply 
\begin{equation} \label{E:f2G}
f_2(s_2(q))=-\frac{2}{15}G(q)-\frac{1}{15}G(-q)+\frac{3}{5}G(q^2),
\end{equation}
where $$G(q)=\Re\left(-\log(q)+240\sum_{n=1}^{\infty}n^2\log(1-q^n)\right).$$
It was also shown in the same theorem that 
\begin{equation} \label{E:FnalG}
G(-q)=9G(q^2)-4G(q^4)-G(q).
\end{equation} 
Substituting \eqref{E:FnalG} into \eqref{E:f2G} yields
\begin{equation} \label{E:f2G2}
f_2(s_2(q))=-\frac{1}{15}G(q)+\frac{4}{15}G(q^4).
\end{equation}

From now on we let $\sigma_3(n)=\displaystyle\sum_{d|n}d^3$, $E_4(q)=1+240\displaystyle\sum_{n=1}^{\infty}\sigma_3(n)q^n,$ the Eisenstein series of weight $4$ for $\Gamma(1)$, $D=q\displaystyle\frac{d}{dq}$, and $\Li_k(z)=\displaystyle\sum_{m=1}^{\infty}\frac{z^m}{m^k}$, the usual polylogarithm function.\\
Since $q\in(0,1)$, it follows by taking differentials in \eqref{E:f2G2} that
\begin{align*} df_2(s_2(q))&=\left(\frac{1}{15}E_4(q)-\frac{16}{15}E_4(q^4)\right)\frac{dq}{q}\\
            &=-\frac{1}{q}+\sum_{n\geq 1}\sigma_3(n)(16q^{n-1}-256q^{4n-1})dq.
\end{align*}
Then we integrate both sides and use the identity $$D^2\left(\Li_3\left(q^{jd}\right)\right)=(jd)^2\Li_1\left(q^{jd}\right),\hspace{10 mm}\, j,d\in\mathbb{N},$$ to recover
\begin{equation}\label{E:f2D} 
\begin{aligned}f_2(s_2(q))&=\Re\left(-2\pi i\tau+\sum_{n\geq 1}\sigma_3(n)\left(16\frac{q^n}{n}-64\frac{q^{4n}}{n}\right)\right)\\
&= \Re\left(-2\pi i\tau+16D^2\left(\sum_{d\geq 1}\Li_3(q^d)-\frac{1}{4}\Li_3(q^{4d})\right)\right).
\end{aligned}
\end{equation}
For $j=1,4$ let $$F_j(\xi)=\displaystyle\sum_{d\geq 1}\Li_3(q^{jd+\xi})=\sum_{d\geq 1}\sum_{m\geq 1}\frac{e^{2\pi i\tau m(jd+\xi)}}{m^3}.$$ It is not hard to see that $F_j(\xi)$ is differentiable at $\xi=0$. Indeed, for any  $\xi\in \left(-\frac{1}{2},\frac{1}{2}\right)$ 
\begin{align*}
\left |\Li_3(q^{jd+\xi})\right |=\left |\sum_{m\geq 1}\frac{e^{2\pi i\tau m(jd+\xi)}}{m^3}\right | = \sum_{m\geq 1}\frac{e^{-2\pi t m(jd+\xi)}}{m^3}\leq \sum_{m\geq 1}\frac{e^{-2\pi t (jd+\xi)}}{m^3}
= e^{-2\pi t (jd+\xi)}\zeta(3),
\end{align*}
where $\zeta$ is the Riemann zeta function.
Since $e^{2\pi t}>1$, it is immediate that $\displaystyle\sum_{d\geq 1}e^{-2\pi t (jd+\xi)}\zeta(3)$ converges. Therefore, it follows by the Weierstrass M-test that $\displaystyle\sum_{d\geq 1}\Li_3(q^{jd+\xi})$ converges uniformly on $\left(-\frac{1}{2},\frac{1}{2}\right)$. It is easily seen that $\Li_3(q^{jd+\xi})$ is differentiable at $\xi=0$ and hence so is $F_j(\xi)$. As a consequence, we have from a basic fact in Fourier analysis (cf.~\cite[Thm.~3.2.1]{Stein}) that the Fourier series of $F_j(\xi)$ converges pointwise to $F_j(\xi)$ at $\xi=0$; i.e., 
$$F_j(0)=\sum_{n\in \mathbb{Z}}\hat{F_j}(n),$$ where $\hat{F_j}(n)$ denote the Fourier coefficients of $F_j$. Following similar computations to those in \cite{BertinMain}, one sees that
\begin{equation*}
\hat{F_j}(n)= \begin{cases}
                 -\displaystyle\frac{1}{2\pi i}\sum_{m\geq 1}\frac{1}{m^3(jm\tau-\frac{n}{4})}         &\text{if } 4|n,\\
                 0                         &\text{otherwise.} 
              \end{cases}\
\end{equation*}
Since $F_j(0)=\displaystyle\sum_{d\geq 1}\Li_3(q^{jd})$ and $D^2=-\displaystyle\frac{1}{4\pi^2}\frac{d^2}{d\tau^2}$, we have from \eqref{E:f2D} that
\begin{align*}
f_2(s_2(q))&=\Re\left(-2\pi i\tau+16D^2\left(F_1(0)-\frac{1}{4}F_4(0)\right)\right)\\
           &=\Re\left(-2\pi i\tau+\frac{8i}{\pi}D^2\left(\sum_{n\in\mathbb{Z}}\sum_{m\geq 1}\frac{1}{m^3}\left(\frac{1}{m\tau+n}-\frac{1}{4(4m\tau+n)}\right)\right)\right)\\
           &=\Re\left(-2\pi i\tau-\frac{4i}{\pi^3}\sum_{n\in\mathbb{Z}}\sum_{m\geq 1}\frac{1}{m}\left(\frac{1}{(m\tau+n)^3}-\frac{4}{(4m\tau+n)^3}\right)\right)\\
           &=\Re\left(-i\left(2\pi \tau+\frac{2}{\pi^3}\sum_{n\in\mathbb{Z}}\sum_{m\neq 0}\frac{1}{m}\left(\frac{1}{(m\tau+n)^3}-\frac{4}{(4m\tau+n)^3}\right)\right)\right)\\
           &=\Im\left(2\pi \tau+\frac{2}{\pi^3}\sum_{n\in\mathbb{Z}}\sum_{m\neq 0}\frac{1}{m}\left(\frac{1}{(m\tau+n)^3}-\frac{4}{(4m\tau+n)^3}\right)\right)\\
    &=\frac{2\Im(\tau)}{\pi^3}\sideset{}{'}\sum_{m,n\in\mathbb{Z}}\biggl(-\left(2\Re\left(\frac{1}{(m\tau+n)^3(m\bar{\tau}+n)}\right)+\frac{1}{[(m\tau+n)(m\bar{\tau}+n)]^2}\right)\\&\qquad+16\left(2\Re\left(\frac{1}{(4m\tau+n)^3(4m\bar{\tau}+n)}\right)+\frac{1}{[(4m\tau+n)(4m\bar{\tau}+n)]^2}\right)\biggr),
\end{align*}
where we have applied the same tricks from \cite{BertinMain} to obtain the last equality. We then use the fact that $2\Re(z)=z+\bar{z}$ for any $z\in \mathbb{C}$ to finish the proof of (i).

One can prove (ii) and (iii) in a similar fashion. For we again have from \cite[Thm. ~2.3]{RogersMain} that under the assumption stated in the proposition 
\begin{align*}
f_3(s_3(q))&=-\frac{1}{8}G(q)+\frac{3}{8}G(q^3),\\
f_4(s_4(q))&=-\frac{1}{3}G(q)+\frac{2}{3}G(q^2).
\end{align*}
\end{proof}

Let us prove some crucial lemmas before establishing Theorem~\ref{T:Main}.
\begin{lemma}\label{L:CM}
If $f,g,$ and $h$ are as defined in Theorem~\ref{T:Main}, then
\begin{align}
f(\tau)&=\sum_{m,n\in\mathbb{Z} }\frac{m^2-2n^2}{2}q^{m^2+2n^2},\label{E:f}\\
g(\tau)&=\sum_{m,n\in\mathbb{Z} }\frac{m^2-3n^2}{2}q^{m^2+3n^2},\label{E:g}\\
h(\tau)&=\sum_{m,n\in\mathbb{Z} }\frac{m^2-4n^2}{2}q^{m^2+4n^2}.\label{E:h}
\end{align}
\end{lemma}
\begin{proof}
We will show \eqref{E:h} first. Let $K=\mathbb{Q}(i)$, $\mathcal{O}_K = \mathbb{Z}[i]$, $\Lambda=(2)\subset \mathcal{O}_K$, and $I(\Lambda)$ = the group of fractional ideals of $\mathcal{O}_K$ coprime to $\Lambda$. Then we define the Hecke Gr\"{o}ssencharacter $\phi:I(\Lambda)\rightarrow \mathbb{C}^{\times}$ of conductor $\Lambda$ by $$\phi((m+in))=(m+in)^2$$ for any $m,n\in \mathbb{Z}$ such that $m$ is odd and $n$ is even, and let 
$$\Psi(\tau):=\sum_{\mathfrak{a}\subseteq\mathcal{O}_K}\phi(\mathfrak{a})q^{N(\mathfrak{a})}=\sum_{k=1}^{\infty}a(k)q^k,$$
where the sum runs through the integral ideals of $\mathcal{O}_K$ coprime to $\Lambda$ and $N(\mathfrak{a})$ denotes the norm of the ideal $\mathfrak{a}$. It then follows from \cite[Thm.~1.31]{Ono} that $\Psi(\tau)$ is a newform in $S_3(\Gamma_0(16),\chi_{-4}).$ Moreover, by \cite[Ex.~1.33]{Ono}, we have that $a(p)=0$ for every prime $p\equiv 3 \pmod 4$, and if $p$ is a prime such that $p=(m+in_0)(m-in_0)=m^2+n_0^2$ for some $m,n_0\in \mathbb{Z}$ with $m$ odd and $n_0=2n$, then $a(p)=2(m^2-4n^2).$ Also, it is clear by the definition of $\Psi$ that $a(k)=0$ for every $k\in\mathbb{N}_\text{even}$. Next, we shall examine $a(k)$ explicitly for each $k\in \mathbb{N}_\text{odd}$. \\

Recall first that since $\Psi(\tau)$ is a Hecke eigenform in $S_3(\Gamma_0(16),\chi_{-4})$, 
\begin{equation}\label{E:eigen}
a(k)a(l)=\sum_{d|(k,l)}\chi_{-4}(d)d^2 a\left(\frac{kl}{d^2}\right)
\end{equation}
holds for all $k,l\in\mathbb{N}$ (cf.~\cite[Ch.~6]{Iwaniec}). If $k$ is odd and all prime factors of $k$ are congruent to $1$ modulo $4$, then it is easily seen by induction that $k=m^2+4n^2$ for some $m,n\in\mathbb{Z}$ with $m$ odd. Now suppose $k$ is odd and $k$ has a prime factor congruent to $3$ modulo $4$, say
$$k=\prod_{p_i\equiv 1\pmod 4}p_i\cdot\prod_{r_j\equiv 3\pmod 4}r_j$$ for some primes $p_i$ and $r_j$. If $\displaystyle\prod_{r_j\equiv 3\pmod 4}r_j$ is a perfect square, then $k$ is again of the form $k=m^2+4n^2$ with $m$ odd. Otherwise, there exists a prime factor $r\equiv 3 \pmod 4$ of $k$ such that $r^l\|k$ for some odd $l$. But then it can be shown inductively using \eqref{E:eigen} that $a(r^l)=0$, so $a(k)$ vanishes in this case. 
Note that for any $k=m^2+4n^2$ with $m$ odd 
\begin{align*}
a(k)&= \begin{cases} \phi((m+2in))+\phi((m-2in)) & \mbox{if } n\neq 0, \\ \phi((m)) & \mbox{if } n=0, \end{cases}\\
&=\begin{cases} 2(m^2-4n^2) & \mbox{if } n\neq 0, \\ m^2 & \mbox{if } n=0. \end{cases}
\end{align*}
Consequently, we may express $\Psi(\tau)$ as 
\begin{equation*}\label{E:Hecke}
\Psi(\tau)=\sum_{k=1}^{\infty}a(k)q^k=\sum_{\substack{m,n\in\mathbb{Z} \\ m \text{  odd}}}\frac{m^2-4n^2}{2}q^{m^2+4n^2}=\sum_{m,n\in\mathbb{Z} }\frac{m^2-4n^2}{2}q^{m^2+4n^2},
\end{equation*}
since $$\sum_{\substack{m,n\in\mathbb{Z} \\ m \text{  even}}}\frac{m^2-4n^2}{2}q^{m^2+4n^2}=0.$$
Computing the first few Fourier coefficients of $\Psi(\tau)$ we see that 
$$\Psi(\tau)=q-6q^5+9q^9+\cdots.$$
On the other hand, we know from \cite{DKM} that 
\begin{equation*}
\eta(4\tau)^6=q-6q^5+9q^9+\cdots\in S_3(\Gamma_0(16),\chi_{-4}).
\end{equation*}
Hence
\begin{equation*}\label{E:Sturm}
h(\tau)=\eta(4\tau)^6=\Psi(\tau)
\end{equation*} by Sturm's theorem (cf.~\cite[Thm.~2.58]{Ono}).

The equalities \eqref{E:f} and \eqref{E:g} can be established in a similar way. Indeed, we see from \cite{Bertin} and \cite{BertinMain} that $f(\tau)$ and $g(\tau)$ are the inverse Mellin transforms of the Hecke $L$-series with respect to some weight 3 Hecke Gr\"{o}ssencharacters defined for the rings $\mathbb{Z}[\sqrt{-2}]$ and $\mathbb{Z}[2\sqrt{-3}],$ respectively.
\end{proof}

\begin{lemma} \label{L:Eta}
If $g_{48}$ and $g$ are as defined in Theorem~\ref{T:Main}, then the following identities hold: 
\begin{align}
g_{48}(\tau)&=\sum_{m,n\in\mathbb{Z}}\left(\left(\frac{m^2-12n^2}{2}\right)q^{m^2+12n^2}+\left(\frac{3m^2-4n^2}{2}\right)q^{3m^2+4n^2}\right),\label{E:EtaPete}\\
g(\tau)+8g(4\tau)&=\sum_{m,n\in\mathbb{Z}}\left(\left(\frac{m^2-12n^2}{2}\right)q^{m^2+12n^2}+\left(\frac{4n^2-3m^2}{2}\right)q^{3m^2+4n^2}\right).\label{E:EtaPeteII}
\end{align}
\end{lemma}\begin{proof}
Let $$\mathfrak{h}_1(\tau):=\sum_{m,n\in\mathbb{Z}}\left(\left(\frac{m^2-12n^2}{2}\right)q^{m^2+12n^2}+\left(\frac{3m^2-4n^2}{2}\right)q^{3m^2+4n^2}\right).$$
Note that by the symmetry of the summation we have
$$\sum_{\substack{m,n\in\mathbb{Z} \\ m \text{  even}}}\left(\left(\frac{m^2-12n^2}{2}\right)q^{m^2+12n^2}+\left(\frac{3m^2-4n^2}{2}\right)q^{3m^2+4n^2}\right)=0.$$
Also, it is obvious that for all $x,y\in\mathbb{Z}$ 
$$x^2+3y^2=\frac{3(x-y)^2+(x+3y)^2}{4}=\frac{3(x+y)^2+(x-3y)^2}{4}.$$
Hence 
\begin{align*}
\mathfrak{h}_1(\tau)&=\sum_{\substack{m\in\mathbb{Z} \\ m \text{  odd}}}\sum_{\substack{n\in\mathbb{Z} \\ n \text{  even}}}\left(\left(\frac{m^2-3n^2}{2}\right)q^{m^2+3n^2}+\left(\frac{3m^2-n^2}{2}\right)q^{3m^2+n^2}\right)
\\&=\sum_{\substack{m\in\mathbb{Z} \\ m \text{  odd}}}\sum_{\substack{n\in\mathbb{Z} \\ n \text{  even}}}\biggl(\left(\frac{m^2-3n^2}{4}\right)q^{\frac{3(m-n)^2+(m+3n)^2}{4}}+\left(\frac{m^2-3n^2}{4}\right)q^{\frac{3(m+n)^2+(m-3n)^2}{4}}\\&\qquad\qquad +\left(\frac{3m^2-n^2}{4}\right)q^{\frac{3(m-n)^2+(3m+n)^2}{4}}+\left(\frac{3m^2-n^2}{4}\right)q^{\frac{3(m+n)^2+(3m-n)^2}{4}}\biggr)\\
&=\sum_{\substack{m>0 \\ m \text{  odd}}}\sum_{\substack{n\in\mathbb{Z} \\ n \text{  even}}}\biggl(\frac{(m-n)(m+3n)}{2}q^{\frac{3(m-n)^2+(m+3n)^2}{4}}+\frac{(m+n)(m-3n)}{2}q^{\frac{3(m+n)^2+(m-3n)^2}{4}}\\&\qquad\qquad +\frac{(m-n)(3m+n)}{2}q^{\frac{3(m-n)^2+(3m+n)^2}{4}}+\frac{(m+n)(3m-n)}{2}q^{\frac{3(m+n)^2+(3m-n)^2}{4}}\biggr)\\
&=\sum_{\substack{m>0 \\ m \text{  odd}}}\sum_{\substack{n>0 \\ n \text{  even}}}\biggl((m-n)(m+3n)q^{\frac{3(m-n)^2+(m+3n)^2}{4}}+(m+n)(m-3n)q^{\frac{3(m+n)^2+(m-3n)^2}{4}}\\&\qquad\qquad +(m-n)(3m+n)q^{\frac{3(m-n)^2+(3m+n)^2}{4}}+(m+n)(3m-n)q^{\frac{3(m+n)^2+(3m-n)^2}{4}}\biggr)\\
& \qquad\qquad+\sum_{\substack{m>0\\ m \text{  odd}}}\left(m^2q^{m^2}+3m^2q^{3m^2}\right).
\end{align*}
Let $\mathcal{A}=\{(k,l)\in \mathbb{N}_\text{odd}^2 \mid l\neq k \mbox{ and } l\neq 3k \}$ and $\mathcal{B}=\mathbb{N}_\text{odd}\times\mathbb{N}_\text{even}.$
Recall that for any $k\in\mathbb{N}$
\begin{equation*}
\chi_{-8}(k)= \begin{cases}
                 1         &\text{if } k\equiv 1,3 \pmod 8,\\
                 -1        &\text{if } k\equiv 5,7 \pmod 8,\\
                 0         &\text{if } k \text{ is even.}
              \end{cases}
\end{equation*}
Thus it is easy to verify that for all $(m,n)\in\mathcal{B}$ the following equalities are true:
$$m-n=\chi_{-8}(|m-n|(m+3n))|m-n|=\chi_{-8}(|m-n|(3m+n))|m-n|,$$
$$m-3n=\chi_{-8}((m+n)|m-3n|)|m-3n|,$$
$$3m-n=\chi_{-8}((m+n)|3m-n|)|3m-n|.$$
Let $(k,l)\in\mathcal{A}$. Then it is obvious that  $(3k^2+l^2)/4 \in\mathbb{N}_\text{odd}$.\\
If $(3k^2+l^2)/4 \equiv 1 \pmod 4$, then either $8|(k-l)$ or $8|(k+l)$, so letting
\begin{equation*}
(m,n)= \begin{cases}
        \left(\displaystyle\frac{3k+l}{4},\frac{|k-l|}{4}\right)         &\text{if } 8|(k-l),\\
        \left(\displaystyle\frac{|3k-l|}{4},\frac{k+l}{4}\right)         &\text{if } 8|(k+l),
        \end{cases}
\end{equation*} yields $(m,n)\in\mathcal{B}$.
Consequently, we have the equality
\begin{multline*}
\left\{(k,l)\in\mathcal{A}\mid \displaystyle\frac{3k^2+l^2}{4}\equiv 1 \pmod 4\right\}=\{(|m-n|,m+3n)\mid (m,n)\in\mathcal{B}\}\\ \dotcup \{(m+n,|m-3n|)\mid (m,n)\in\mathcal{B}\},
\end{multline*}
where $\dotcup$ denotes disjoint union, since the inclusion $\supseteq$ is obvious.\\ 
If $(3k^2+l^2)/4 \equiv 3 \pmod 4$, then either $8|(3k-l)$ or $8|(3k+l)$. Hence, if we let 
\begin{equation*}
(m,n)= \begin{cases}
        \left(\displaystyle\frac{k+l}{4},\frac{|3k-l|}{4}\right)         &\text{if } 8|(3k-l),\\
        \left(\displaystyle\frac{|k-l|}{4},\frac{3k+l}{4}\right)         &\text{if } 8|(3k+l),
        \end{cases}
\end{equation*}
then $(m,n)\in\mathcal{B},$ so 
\begin{multline*}
\left\{(k,l)\in\mathcal{A}\mid \displaystyle\frac{3k^2+l^2}{4}\equiv 3 \pmod 4\right\}=\{(|m-n|,3m+n)\mid (m,n)\in\mathcal{B}\}\\\dotcup \{(m+n,|3m-n|)\mid (m,n)\in\mathcal{B}\}.
\end{multline*}
Therefore, we can simplify the last expression of $\mathfrak{h}_1(\tau)$ above to obtain
\begin{equation*}
\mathfrak{h}_1(\tau)=\sum_{(k,l)\in\mathcal{A}}\chi_{-8}(kl)kl q^{\frac{3k^2+l^2}{4}}+\sum_{\substack{m>0\\ m \text{  odd}}}\left(m^2q^{m^2}+3m^2q^{3m^2}\right)=\sum_{m,n\in\mathbb{N}}\chi_{-8}(mn)mn q^{\frac{3m^2+n^2}{4}}.
\end{equation*}
Then \eqref{E:EtaPete} follows easily since $$g_{48}(\tau)=\left(\frac{\eta(4\tau)^9}{\eta(2\tau)^3\eta(8\tau)^3}\right)\left(\frac{\eta(12\tau)^9}{\eta(6\tau)^3\eta(24\tau)^3}\right)$$ and the following identity holds \cite[Prop.~1.6]{Kohler}:
\begin{equation*}
\frac{\eta(2\tau)^9}{\eta(\tau)^3\eta(4\tau)^3}=\sum_{n\in \mathbb{N}}\chi_{-8}(n)n q^{\frac{n^2}{8}}.
\end{equation*}

Now, let $$\mathfrak{h}_2(\tau):=\sum_{m,n\in\mathbb{Z}}\left(\left(\frac{m^2-12n^2}{2}\right)q^{m^2+12n^2}+\left(\frac{4n^2-3m^2}{2}\right)q^{3m^2+4n^2}-(4m^2-12n^2)q^{4m^2+12n^2}\right).$$
Then it is easy to see that 
$$\mathfrak{h}_2(\tau)=\sum_{\substack{m\in\mathbb{Z} \\ m \text{  odd}}}\sum_{\substack{n\in\mathbb{Z} \\ n \text{  even}}}\left(\left(\frac{m^2-3n^2}{2}\right)q^{m^2+3n^2}+\left(\frac{n^2-3m^2}{2}\right)q^{3m^2+n^2}\right).$$ 
Repeating the arguments above and using the fact that for every $(m,n)\in\mathcal{B}$
\begin{align*}
m-n &= \chi_{-4}\left(|m-n|(m+3n)\right)|m-n|,\\
n-m &= \chi_{-4}\left(|n-m|(3m+n)\right)|n-m|,\\
m-3n &= \chi_{-4}\left((m+n)|m-3n|\right)|m-3n|,\\
n-3m &= \chi_{-4}\left((m+n)|n-3m|\right)|n-3m|,\\
\end{align*}
we can deduce that 
\begin{equation*}
\mathfrak{h}_2(\tau)=\sum_{m,n\in\mathbb{N}}\chi_{-4}(mn)mn q^{\frac{3m^2+n^2}{4}}.
\end{equation*}
We then employ the $q$-series identity \cite[Cor.~1.4]{Kohler}
\begin{equation*}
\eta(\tau)^3=\sum_{n\in \mathbb{N}}\chi_{-4}(n)n q^{\frac{n^2}{8}}
\end{equation*}
to conclude that 
\begin{equation*}
g(\tau)=\eta(2\tau)^3 \eta(6\tau)^3=\mathfrak{h}_2(\tau).
\end{equation*}
By \eqref{E:g}, we see that 
$$\sum_{m,n\in\mathbb{Z}}(4m^2-12n^2)q^{4m^2+12n^2}=8g(4\tau),$$
so \eqref{E:EtaPeteII} follows.
\end{proof}

\begin{lemma} \label{L:Finale}
If $g_{24}^{(1)},g_{24}^{(2)}$ and $g_{40}$ are as defined in Theorem~\ref{T:Main} and $s\in\mathbb{C}$ with $\Re(s)>2$, then the following identities hold: 
\begin{align}
L(g_{24}^{(1)},s)&=\frac{1}{2}\sideset{}{'}\sum_{m,n\in\mathbb{Z}}\left(\frac{m^2-6n^2}{(m^2+6n^2)^{s}}+\frac{2m^2-3n^2}{(2m^2+3n^2)^{s}}\right),\label{E:Bertin}\\
L(g_{24}^{(2)},s)&=\frac{1}{2}\sideset{}{'}\sum_{m,n\in\mathbb{Z}}\left(\frac{m^2-6n^2}{(m^2+6n^2)^{s}}+\frac{3m^2-2n^2}{(3m^2+2n^2)^{s}}\right),\label{E:BertinT}\\
L(g_{40},s)&=\frac{1}{2}\sideset{}{'}\sum_{m,n\in\mathbb{Z}}\left(\frac{m^2-10n^2}{(m^2+10n^2)^{s}}+\frac{5m^2-2n^2}{(5m^2+2n^2)^{s}}\right).\label{E:BertinTh}
\end{align}
\end{lemma}
\begin{proof}
We have immediately from the proof of \cite[Thm.~4.1]{BertinMain} that
$$L_{\mathbb{Q}(\sqrt{-6})}(\phi,s)=\frac{1}{2}\sideset{}{'}\sum_{m,n\in\mathbb{Z}}\left(\frac{m^2-6n^2}{(m^2+6n^2)^{s}}+\frac{3m^2-2n^2}{(3m^2+2n^2)^{s}}\right),$$ where $\phi$ is the Hecke Gr\"{o}ssencharacter given by 
\begin{align*}
\phi((m+n\sqrt{-6}))&=(m+n\sqrt{-6})^2,\\
\phi((2,\sqrt{-6}))&=-2,
\end{align*}
for any $m,n\in\mathbb{Z}.$ Considering the first terms of this Hecke $L$-series, one sees that its inverse Mellin transform is exactly $g_{24}^{(2)}(\tau)$ by Sturm's theorem.
Similarly, if we define the Hecke Gr\"{o}ssencharacter $\psi$ by 
\begin{align*}
\psi((m+n\sqrt{-6}))&=(m+n\sqrt{-6})^2,\\
\psi((2,\sqrt{-6}))&=2,
\end{align*}
then we obtain the Hecke $L$-series
$$L_{\mathbb{Q}(\sqrt{-6})}(\psi,s)=\frac{1}{2}\sideset{}{'}\sum_{m,n\in\mathbb{Z}}\left(\frac{m^2-6n^2}{(m^2+6n^2)^{s}}+\frac{2m^2-3n^2}{(2m^2+3n^2)^{s}}\right),$$
whose inverse Mellin transform is $g_{24}^{(1)}(\tau)$. Consequently, \eqref{E:Bertin} and \eqref{E:BertinT} follow.

To show \eqref{E:BertinTh} we shall imitate the proof of \cite[Thm.~4.1]{BertinMain}. 
Recall that in $\mathbb{Z}(\sqrt{-10})$ there are two classes of ideals, namely $$\mathcal{A}_0=\left\{(m+n\sqrt{-10})\mid m,n\in \mathbb{Z}\right\} \text{ and } \mathcal{A}_1=\left\{(m+n\sqrt{-10})\mathcal{P}\mid m,n\in \mathbb{Z}\right\},$$ where $\mathcal{P}=(2,\sqrt{-10}).$ Defining the Hecke character $$\phi((m+n\sqrt{-10}))=(m+n\sqrt{-10})^2 , \hspace{5mm}\phi(\mathcal{P})=-2 $$
and applying the formula
$$L_F(\phi,s)=\sum_{cl(P)}\frac{\phi(P)}{N(P)^{2-s}}\left(\frac{1}{2}\sideset{}{'}\sum_{\lambda\in P}\frac{\bar{\lambda}^2}{\left(\lambda\bar{\lambda}\right)^s}\right),$$
we have
\begin{equation*}
L_{\mathbb{Q}(\sqrt{-10})}(\phi,s)=\frac{1}{2}\sideset{}{'}\sum_{m,n\in\mathbb{Z}}\left(\frac{m^2-10n^2}{(m^2+10n^2)^{s}}+\frac{5m^2-2n^2}{(5m^2+2n^2)^{s}}\right),
\end{equation*}
and the inverse Mellin transform of this Hecke $L$-series equals $g_{40}(\tau).$ Since the conductors of the Hecke characters defined above are trivial and the discriminants of $\mathbb{Q}(\sqrt{-6})$ and $\mathbb{Q}(\sqrt{-10})$ are $-24$ and $-40$, respectively, we have that $g_N$ are newforms of weight 3 and level $N$ having CM by $\chi_{-N}$ (cf. \cite[\S 1]{Schutt}).
\end{proof}
\begin{lemma}\label{L:FinaleII}
Let $t\in\mathbb{C}$ be such that $\Re(t)>1.$ Then the following equalities hold:
\begin{align}
2\left(1-\frac{3}{2^t}+\frac{2}{2^{2t}}\right)\zeta(t)L(\chi_{-4},t)&=\sideset{}{'}\sum_{m,n\in\mathbb{Z}}\left(\frac{1}{(m^2+4n^2)^t}-\frac{1}{(2m^2+2n^2)^t}\right)\label{E:First},\\
2L(\chi_{8},t)L(\chi_{-3},t)&=\sideset{}{'}\sum_{m,n\in\mathbb{Z}}\left(\frac{1}{(m^2+6n^2)^t}-\frac{1}{(2m^2+3n^2)^t}\right)\label{E:Lattice},\\
2L(\chi_5,t)L(\chi_{-8},t)&=\sideset{}{'}\sum_{m,n\in\mathbb{Z}}\left(\frac{1}{(m^2+10n^2)^t}-\frac{1}{(2m^2+5n^2)^t}\right)\label{E:Third},\\
2L(\chi_{12},t)L(\chi_{-4},t)&=\sideset{}{'}\sum_{m,n\in\mathbb{Z}}\left(\frac{1}{(m^2+12n^2)^t}-\frac{1}{(3m^2+4n^2)^t}\right)\label{E:Zagier},\\
2L(\chi_{24},t)L(\chi_{-3},t)&=\sideset{}{'}\sum_{m,n\in\mathbb{Z}}\left(\frac{1}{(m^2+18n^2)^t}-\frac{1}{(2m^2+9n^2)^t}\right)\label{E:Fifth}.
\end{align}
\end{lemma}
\begin{proof}
First, recall from \cite[\S IV]{Lattice} that if we set
$$S(a,b,c;t):=\sideset{}{'}\sum_{m,n\in\mathbb{Z}}\frac{1}{\left(am^2+bmn+cn^2\right)^t},$$then the following equalities hold:
\begin{align*}
S(1,0,1;t)&=4\zeta(t)L(\chi_{-4},t),\\
S(1,0,4;t)&=2\left(1-2^{-t}+2^{1-2t}\right)\zeta(t)L(\chi_{-4},t),\\
S(1,0,6;t)&=\zeta(t)L(\chi_{-24},t)+L(\chi_8,t)L(\chi_{-3},t),\\
S(1,0,10;t)&=\zeta(t)L(\chi_{-40},t)+L(\chi_5,t)L(\chi_{-8},t),\\
S(1,0,12;t)&=\left(1+2^{-2t}+2^{2-4t}\right)\zeta(t)L(\chi_{-3},t)+L(\chi_{12},t)L(\chi_{-4},t),\\
S(1,0,18;t)&=\left(1-2\cdot 3^{-t}+3^{1-2t}\right)\zeta(t)L(\chi_{-8},t)+L(\chi_{24},t)L(\chi_{-3},t).
\end{align*}
We will exhibit how to prove \eqref{E:Lattice} only, since the other identities can be shown similarly. Let $Q_1$ and $Q_2$ be the quadratic forms of discriminant $-24$ given by $$Q_1(m,n)=m^2+6n^2,\,\, Q_2(m,n)=2m^2+3n^2,$$ and for each $j\in \{1,2\}$ and $k\in \mathbb{N}$ let 
$$R_{Q_j}(k)=\#\left\{(m,n)\in\mathbb{Z}^2 \mid Q_j(m,n)=k \right\}.$$ By the formulas above, we see that 
\begin{equation} \label{E:Zucker}
\sum_{k=1}^{\infty}\frac{R_{Q_1}(k)}{k^t}=\zeta(t)L(\chi_{-24},t)+L(\chi_8,t)L(\chi_{-3},t).
\end{equation}
Notice that, for any given $l\in\mathbb{N}$, $2m^2+3n^2=2l$ is equivalent to $m^2+6b^2=l$, where $n=2b$. This implies that $R_{Q_2}(2l)=R_{Q_1}(l)$. Similarly, it can be checked that $R_{Q_2}(3l)=R_{Q_1}(l)$ and $R_{Q_2}(6l)=R_{Q_2}(l)$.
As a result, we have
\begin{equation}\label{E:Finale}
\begin{aligned}
\sideset{}{'}\sum_{m,n\in\mathbb{Z}}\frac{1}{(2m^2+3n^2)^t}&=\sum_{k=1}^{\infty}\frac{R_{Q_2}(k)}{k^t}\\
&=\sum_{\substack{k=1 \\ (k,6)=1}}^{\infty}\frac{R_{Q_2}(k)}{k^t}+\sum_{\substack{k=1 \\ 2|k}}^{\infty}\frac{R_{Q_2}(k)}{k^t}+\sum_{\substack{k=1 \\ 3|k}}^{\infty}\frac{R_{Q_2}(k)}{k^t}-\sum_{\substack{k=1 \\ 6|k}}^{\infty}\frac{R_{Q_2}(k)}{k^t}\\
&=\sum_{\substack{k=1 \\ (k,6)=1}}^{\infty}\frac{R_{Q_2}(k)}{k^t}+\left(\frac{1}{2^t}+\frac{1}{3^t}\right)\sum_{k=1}^{\infty}\frac{R_{Q_1}(k)}{k^t}-\frac{1}{6^t}\sum_{k=1}^{\infty}\frac{R_{Q_2}(k)}{k^t}.
\end{aligned}
\end{equation}
If $(k,6)=1$, then  
\begin{equation*}
k \equiv \begin{cases}
                 1  \,\,\,\,\,\pmod 3         &\text{if } k=Q_1(m,n),\\
                 -1 \pmod 3        &\text{if } k=Q_2(m,n).
              \end{cases}\
\end{equation*}
Hence we find from the well-known formula due to Dirichlet \cite[p.~229]{Dirichlet} that $$R_{Q_2}(k)=\left(1-\chi_{-3}(k)\right)\sum_{l|k}\chi_{-24}(l)=\sum_{l|k}\chi_{-24}(l)-\sum_{l|k}\chi_{-3}\left(\frac{k}{l}\right)\chi_{8}(l).$$
It follows that 
\begin{equation}\label{E:LatticeR}
\begin{aligned}
\sum_{\substack{k=1 \\ (k,6)=1}}^{\infty}\frac{R_{Q_2}(k)}{k^t}&=\sum_{\substack{k=1 \\ (k,6)=1}}^{\infty}\frac{\left(\textbf{1}\ast\chi_{-24}\right)(k)}{k^t}-\sum_{\substack{k=1 \\ (k,6)=1}}^{\infty}\frac{\left(\chi_{-3}\ast\chi_8\right)(k)}{k^t}\\
&=\left(1-\frac{1}{2^t}\right)\left(1-\frac{1}{3^t}\right)\zeta(t)L(\chi_{-24},t)\\
&\qquad\qquad -\left(1+\frac{1}{2^t}\right)\left(1+\frac{1}{3^t}\right)L(\chi_{-3},t)L(\chi_8,t),
\end{aligned}
\end{equation}
where $\ast$ denotes the Dirichlet convolution. Then \eqref{E:Lattice} can be derived easily using \eqref{E:Zucker}, \eqref{E:Finale}, and \eqref{E:LatticeR}.
\end{proof}

We are now in a good position to prove our main theorem.
\begin{proof}[Proof of Theorem~\ref{T:Main}]
Applying Lemma~\ref{L:s2}, Proposition~\ref{P:general}(i) for $\tau\in\{\frac{\sqrt{-1}}{2},\frac{\sqrt{-3}}{2} \}$, Lemma~\ref{L:CM},  Lemma~\ref{L:Eta}, and Lemma~\ref{L:FinaleII}, we have immediately that 
\begin{align*}
f_2(64)&=\frac{1}{\pi^3}\sideset{}{'}\sum_{m,n\in\mathbb{Z}}\biggl(-\left(\frac{256n^2}{(m^2+4n^2)^3}-\frac{16}{(m^2+4n^2)^2}\right)+16\left(\frac{4n^2}{(4m^2+n^2)^3}-\frac{1}{(4m^2+n^2)^2}\right)\biggr)\\
&=\frac{128}{\pi^3}\left(\frac{1}{2}\sideset{}{'}\sum_{m,n\in\mathbb{Z}}\frac{m^2-4n^2}{(m^2+4n^2)^3}\right)\\&=\frac{128}{\pi^3}L(h,3),\\
f_2(256)&=\frac{\sqrt{3}}{\pi^3}\sideset{}{'}\sum_{m,n\in\mathbb{Z}}\biggl(-\left(\frac{256n^2}{(3m^2+4n^2)^3}-\frac{16}{(3m^2+4n^2)^2}\right)\\&\qquad\qquad+16\left(\frac{4n^2}{(12m^2+n^2)^3}-\frac{1}{(12m^2+n^2)^2}\right)\biggr)\\
&=\frac{64\sqrt{3}}{\pi^3}\left(\frac{1}{2}\sideset{}{'}\sum_{m,n\in\mathbb{Z}}\left(\frac{m^2-12n^2}{(m^2+12n^2)^3}+\frac{3m^2-4n^2}{(3m^2+4n^2)^3}\right)\right)\\&\qquad\qquad+\frac{16\sqrt{3}}{\pi^3}\sideset{}{'}\sum_{m,n\in\mathbb{Z}}\left(\frac{1}{(m^2+12n^2)^2}-\frac{1}{(3m^2+4n^2)^2}\right)\\
&=\frac{64\sqrt{3}}{\pi^3}L(g_{48},3)+\frac{16}{3\pi}L(\chi_{-4},2),
\end{align*}
where we have used the fact that $L(\chi_{12},2)=\frac{\pi^2}{6\sqrt{3}}$ to get the last equality.\\

Similarly, using Proposition~\ref{P:general} and the lemmas in this section properly, we get
\begin{align*}
f_3(216)&=\frac{45\sqrt{6}}{\pi^3}\left(\frac{1}{2}\sideset{}{'}\sum_{m,n\in\mathbb{Z}}\left(\frac{m^2-6n^2}{(m^2+6n^2)^3}+\frac{2m^2-3n^2}{(2m^2+3n^2)^3}\right)\right)\\&\qquad\qquad+\frac{45\sqrt{6}}{4\pi^3}\sideset{}{'}\sum_{m,n\in\mathbb{Z}}\left(\frac{1}{(m^2+6n^2)^2}-\frac{1}{(2m^2+3n^2)^2}\right)\\
&=\frac{45\sqrt{6}}{\pi^3}L(g_{24}^{(1)},3)+\frac{45\sqrt{3}}{16\pi}L(\chi_{-3},2),\\
f_3(1458)&=\frac{90\sqrt{3}}{\pi^3}\left(\frac{1}{2}\sideset{}{'}\sum_{m,n\in\mathbb{Z}}\left(\frac{m^2-12n^2}{(m^2+12n^2)^3}+\frac{4m^2-3n^2}{(4m^2+3n^2)^3}\right)\right)\\&\qquad\qquad+\frac{45\sqrt{3}}{2\pi^3}\sideset{}{'}\sum_{m,n\in\mathbb{Z}}\left(\frac{1}{(m^2+12n^2)^2}-\frac{1}{(4m^2+3n^2)^2}\right)\\
&=\frac{810\sqrt{3}}{8\pi^3}L(g,3)+\frac{15}{2\pi}L(\chi_{-4},2),\\
f_4(648)&=\frac{160}{\pi^3}\left(\frac{1}{2}\sideset{}{'}\sum_{m,n\in\mathbb{Z}}\frac{m^2-4n^2}{(m^2+4n^2)^3}\right)+\frac{40}{\pi^3}\sideset{}{'}\sum_{m,n\in\mathbb{Z}}\left(\frac{1}{(m^2+4n^2)^2}-\frac{1}{(2m^2+2n^2)^2}\right)\\
&=\frac{160}{\pi^3}L(h,3)+\frac{5}{\pi}L(\chi_{-4},2),\\
f_4(2304)&=\frac{80\sqrt{6}}{\pi^3}\left(\frac{1}{2}\sideset{}{'}\sum_{m,n\in\mathbb{Z}}\left(\frac{m^2-6n^2}{(m^2+6n^2)^3}+\frac{3m^2-2n^2}{(3m^2+2n^2)^3}\right)\right)\\&\qquad\qquad+\frac{20\sqrt{6}}{\pi^3}\sideset{}{'}\sum_{m,n\in\mathbb{Z}}\left(\frac{1}{(m^2+6n^2)^2}-\frac{1}{(2m^2+3n^2)^2}\right)\\
&=\frac{80\sqrt{6}}{\pi^3}L(g_{24}^{(2)},3)+\frac{5\sqrt{3}}{\pi}L(\chi_{-3},2),\\
f_4(20736)&=\frac{80\sqrt{10}}{\pi^3}\left(\frac{1}{2}\sideset{}{'}\sum_{m,n\in\mathbb{Z}}\left(\frac{m^2-10n^2}{(m^2+10n^2)^3}+\frac{5m^2-2n^2}{(5m^2+2n^2)^3}\right)\right)\\&\qquad\qquad+\frac{20\sqrt{10}}{\pi^3}\sideset{}{'}\sum_{m,n\in\mathbb{Z}}\left(\frac{1}{(m^2+10n^2)^2}-\frac{1}{(5m^2+2n^2)^2}\right)\\
&=\frac{80\sqrt{10}}{\pi^3}L(g_{40},3)+\frac{32\sqrt{2}}{5\pi}L(\chi_{-8},2),\\
f_4(614656)&=\frac{800\sqrt{2}}{3\pi^3}\left(\frac{1}{2}\sideset{}{'}\sum_{m,n\in\mathbb{Z}}\frac{m^2-2n^2}{(m^2+2n^2)^3}\right)+\frac{60\sqrt{2}}{\pi^3}\sideset{}{'}\sum_{m,n\in\mathbb{Z}}\left(\frac{1}{(m^2+18n^2)^2}-\frac{1}{(2m^2+9n^2)^2}\right)\\
&=\frac{800\sqrt{2}}{3\pi^3}L(f,3)+\frac{10\sqrt{3}}{\pi}L(\chi_{-3},2),
\end{align*}
since $\zeta(2)=\frac{\pi^2}{6},L(\chi_8,2)=\frac{\pi^2}{8\sqrt{2}},L(\chi_5,2)=\frac{4\pi^2}{25\sqrt{5}},$ and $L(\chi_{24},2)=\frac{\pi^2}{4\sqrt{6}}.$
The latter equalities in \eqref{A:64}-\eqref{A:614656} can be deduced using the following functional equations:
\begin{equation}\label{E:FE}
\left(\frac{\sqrt{N}}{2\pi}\right)^s\Gamma(s)L(f,s)=\pm\left(\frac{\sqrt{N}}{2\pi}\right)^{3-s}\Gamma(3-s)L(f,3-s),
\end{equation}
$$\left(\frac{\pi}{k}\right)^{-\frac{2-s}{2}}\Gamma\left(\frac{2-s}{2}\right)L(\chi_{-k},1-s)=\left(\frac{\pi}{k}\right)^{-\frac{s+1}{2}}\Gamma\left(\frac{s+1}{2}\right)L(\chi_{-k},s),$$
where $f$ is a newform in $S_3(\Gamma_0(N),\chi)$ with real Fourier coefficients and $\chi$ is a real character. (When $f$ is any of the newforms given in Theorem~\ref{T:Main}, the sign of the functional equation \eqref{E:FE} is `$+$' instead of `$\pm$' by numerical approximation.)
\end{proof}

\section{Mahler measures of other Laurent polynomials} \label{sec:other}
In this section, we show that some other interesting formulas can be deduced easily from the results in Theorem~\ref{T:Main}. More precisely, let us consider the Mahler measures of a family of Laurent polynomials 
$$Q_k:=x+\frac{1}{x}+y+\frac{1}{y}+z+\frac{1}{z}+xy+\frac{1}{xy}+yz+\frac{1}{yz}+xz+\frac{1}{xz}+xyz+\frac{1}{xyz}-k,$$ where $k\in\mathbb{C}$, studied by Bertin in \cite{Bertin} and \cite{BertinMain}. Her results include the following formulas:
\begin{align}
2m(Q_{-36})&=4m(Q_{-6})+m(Q_0),\label{E:B1}\\
m(Q_0)&=\frac{12\sqrt{3}}{\pi^3}L(g,3)=2L'(g,0),\label{E:B2}\\
m(Q_{12})&=4m(Q_{0})=8L'(g,0).\label{E:B3}
\end{align}
On the other hand, Rogers \cite[Thm.~2.5]{RogersMain} proved that 
\begin{theorem}\label{T:RogersII}
For $|z|$ sufficiently large, 
\begin{equation*}
m(Q_{z-4})=-\frac{1}{15}f_3\left(\frac{(16-z)^3}{z^2}\right)+\frac{8}{15}f_3\left(-\frac{(4-z)^3}{z}\right).
\end{equation*}
\end{theorem}
Plugging in the value $z=16$ in the theorem above and applying \eqref{E:B3} yield 
\begin{equation}\label{E:R1}
f_3(108)=15L'(g,0).
\end{equation} 
Alternatively, one can derive this formula directly using similar arguments in the proof of Theorem~\ref{T:Main}. We employ Theorem~\ref{T:RogersII} again when $z=-32$ and use \eqref{A:1458}, \eqref{E:B1}, \eqref{E:B2}, and \eqref{E:R1} to obtain
\begin{corollary}\label{C:P1}
Let $Q_k$ be the Laurent polynomial defined above. Then the following equalities hold:
\begin{align*}
m(Q_{-36})&=2(4L'(g,0)+L'(\chi_{-4},-1)),\\
m(Q_{-6})&=\frac{1}{2}(7L'(g,0)+2L'(\chi_{-4},-1)).
\end{align*}
\end{corollary}

\section{Conclusion}\label{Sec:Conclusion}
Besides the results above, we have found many other conjectured formulas by computing numerically to high accuracy using \texttt{Maple}; e.g., 
\begin{align*}
f_2(-64)&\stackrel{?}=2(L'(g_{32},0)+L'(\chi_{-4},-1)),\\
f_2(-512)&\stackrel{?}=L'(g_{64},0)+L'(\chi_{-8},-1),\\
f_4(-1024)&\stackrel{?}=\frac{8}{5}(5L'(g_{20},0)+2L'(\chi_{-4},-1)),\\
f_4(-12288)&\stackrel{?}=\frac{40}{9}(L'(g_{36},0)+2L'(\chi_{-3},-1)),\\
f_4(-82944)&\stackrel{?}=\frac{40}{13}(L'(g_{52},0)+2L'(\chi_{-4},-1)),
\end{align*}
where $g_N$ is a newform in $S_3(\Gamma_1(N))$ with rational coefficients. (Here $\stackrel{?}=$ means that they are equal to at least 70 decimal places.) As a short-term project, it
might be of interest to rigously prove these conjectured formulas though one might need more time to investigate why the formulas of this type make sense.\\

\noindent\textbf{Acknowledgements}\\

The author would like to thank Matthew Papanikolas for pointing out the numerical evidence of the first formula in Corollary~\ref{C:Hyper}, which chiefly inspires the author to write this paper, and many helpful discussions. The author is also grateful to Mathew Rogers for useful advice and suggestions. Finally, the author thanks Bruce Berndt for directing him to reference \cite{BM}.
\bibliographystyle{amsplain}

\end{document}